\documentclass{article}
\usepackage{amsmath,amsthm,amssymb,mathrsfs,amstext, titlesec,enumitem, comment, graphicx, color, xcolor, inputenc, comment}
\usepackage{hyperref}
\usepackage{xpatch}
\makeatletter
\def\Ddots{\mathinner{\mkern1mu\raise\p@
\vbox{\kern7\p@\hbox{.}}\mkern2mu
\raise4\p@\hbox{.}\mkern2mu\raise7\p@\hbox{.}\mkern1mu}}
\makeatother

\titleformat{\section}
{\normalfont\Large\bfseries}{\thesection}{1em}{}
\titleformat*{\subsection}{\Large\bfseries}
\titleformat*{\subsubsection}{\large\bfseries}
\titleformat*{\paragraph}{\large\bfseries}
\titleformat*{\subparagraph}{\large\bfseries}

\makeatother
\newtheorem{theorem}{Theorem}[section]
\newtheorem{lemma}[theorem]{Lemma}

\theoremstyle{definition}
\newtheorem{definition}[theorem]{Definition}

\newcommand{\bN}{{\mathbb N}}

\newcommand{\pf}{{\mathcal P}_f}
\newcommand{\ntos}{{}^{\hbox{$\bN$}}\!S}

\date{\vspace{-5ex}}

\begin{document}

\title{Cartesian products of two $CR$ sets}
\author{Sayan Goswami\\   \textit{ sayan92m@gmail.com}\footnote{Ramakrishna Mission Vivekananda Educational and Research Institute, Belur Math,
Howrah, West Benagal-711202, India.}}

\maketitle
\begin{abstract}

The notions of CR set is intimately related with the generalized van der Waerden's theorem. In this article, we prove the product of two CR sets is again a CR set. This  answers  \cite[Question 4.2.]{H}. We use combinatorial arguments to prove our result.

\end{abstract}
Subject class: 05D10,  22A15, 54D35

\section{Introduction}

%$A \subseteq \mathbb{N}$ is a similis set if there exists a finite set $D$ for which $\mathbb{N} = \sqcup_{d \in D}dA$,

For any nonempty set $X$, let $\mathcal{P}_f(X)$ be the set of all nonempty finite subsets of $X.$
In Arithmetic Ramsey theory, people deal with the monochromatic patterns found in any given finite coloring of the
integers or of the natural numbers $\bN$. Here, ``coloring” means disjoint partition, and a set is called ``monochromatic” if it is included in one part of the partition. A cornerstone result in this field of
research is Van der Waerden's Theorem \cite{14}, which states that for any finite coloring of the set of natural numbers (usually denoted by $\mathbb{N}$), one always finds monochromatic arithmetic progressions of arbitrary length. One can show that if a set $A$ contains arithmetic progressions of arbitrary length (A.P. rich in short), then if $A=B\cup C$ is a 2-coloring of $A$, one of $B$ and $C$ is A.P. rich\footnote{though not relevant with this article, it should be noted that the polynomial extension of this beautiful result has been recently found in \cite[Theorem 19]{sg1}}. This property is called partition regularity.

 The notions of {\it piecewise syndetic} set have been extensively studied in Ramsey theory and its connection with the Algebra in the Stone-\v{C}ech Compactification of $\bN.$ For every $A\subseteq \bN,$ and $n\in \bN$, denote $-n+A=\{m:m+n\in A\}.$
 \begin{definition}[Piecewise syndetic]
     A set $A\subseteq \bN$ is said to be piecewise syndetic if there exists a finite set $F\subset \bN$ such that for every finite subset $G$ of $\bN$, there exists $y\in \bN$ such that $G+y\subset \cup_{t\in F}(-t+A)$.
 \end{definition}
 It can be proven that piecewise syndetic sets are partition regular. For general semigroups, one can define this notion naturally. For details, readers can see \cite{HS}. In \cite{b2}, authors introduced a more general notion of large sets: combinatorially rich sets for commutative semigroups. Later in \cite{H}, authors introduced the notion of combinatorially rich sets for general semigroups and proved that these sets are partition regular. It can be easily proved that if the semigroup is commutative, then these sets are A.P. rich.
%Let us recall the following definition of combinatorially rich sets from .

\begin{definition}[$CR$ set] \cite[Definition 2.2]{H}
Let $(S,\cdot)$ be an arbitrary semigroup. A set $A\subseteq S$ is said to be a {\it combinatorially rich} set ($CR$-set) if and only
if for each $k\in \mathbb{N}$, there exists $r\in \mathbb{N}$ such that whenever $F\in\pf(\,\ntos\,)$ with $|F|\leq k$, there exist $m\in\bN$, 
$a\in S^{m+1}$, and $t(1)<t(2)<\ldots<t(m)\leq r$ in $\bN$ such that
for each $f\in F$, $$a(1)\cdot f\big(t(1)\big)\cdot a(2)\cdot f\big(t(2)\big)\cdot a(3)\cdots a(m)\cdot f\big(t(m)\big)\cdot a(m+1)\in A.$$
Call $r=r(A,k).$
\end{definition}

Using the algebra of the Stone-\v{C}ech compactification of discrete semigroups, in \cite{HSa}, authors proved several large sets are preserved under taking cartesian products. For combinatorial proofs, we refer to the article \cite{sg}.
Inspired by these studies, in \cite[Question 4.2.]{H}, authors asked if there exist two infinite semigroups $S$ and $T$, and two CR-sets $A$ and $B$ in $S$ and $T$ respectively, such that $A\times B$ is not a CR-set in $S\times T$. Here we show that such a possibility can't occur. We prove the following theorem:

\begin{theorem}\label{m}
Let $S$ and $T$ be semigroups, let $A$ be a $CR$ set in $S$, and let $B$ be a $CR$ set in $T$. Then $A\times B$ is a $CR$ set in $S\times T$.
\end{theorem}

\section{Proof of our result}

%To prove our main result, first, we need two following lemmas. 
$IP$ sets play a major role in Ramsey theory due to their partition regularity, as proved in \cite{21}. $IP$ sets can be defined over any semigroup with a little modification if needed. Here we define the notions of $IP$ sets over $\pf(\bN)$.
\begin{definition}[$IP$ set over $\pf(\bN)$]
A set $A\subseteq \pf(\bN)$ is said to be an $IP$ set if there is a sequence $\langle H_n\rangle_n$ in $\pf(\bN)$ such that
\begin{enumerate}
    \item[$(1)$] for every $n\in \bN,$ we have $\max H_n<\min H_{n+1},$ and
    \item[$(2)$] the set $A=FU(\langle H_n\rangle_n)=\left\lbrace\bigcup_{n\in K}H_n:K\in \pf(\bN)\right\rbrace.$ 
\end{enumerate}
\end{definition}
A consequence of Hindman's theorem \cite{21} says that for every finite coloring of $\pf(\bN)$, there exists a monochromatic $IP$ set. In our proof, we need the finitary version of $IP$ sets, known as $IP_r$ sets, that we define below.

\begin{definition}[$IP_r$ set  over $\pf(\bN)$]
    For every $r\in \bN$, a set $B\subseteq \pf(\bN)$  is called an $IP_r$ set if there exists sequence $\langle H_n\rangle_{n=1}^r$ in $\pf(\bN)$ such that
  \begin{enumerate}
      \item[$(1)$] for every $n\in \bN,$ we have $\max H_n<\min H_{n+1},$ and 
      \item[$(2)$] the set $B=FU(\langle H_n\rangle_{n=1}^r)=\left\lbrace\bigcup_{n\in K}H_n:\emptyset \neq K\subseteq \{1,2,\ldots ,r\}\right\rbrace.$
  \end{enumerate}  
\end{definition}

 A set is called $IP^*$ (resp. $IP_r^*$) if that set intersects with every $IP$ set (resp. $IP_r$ sets). Let us recall the following lemma, which is essential for our purpose.
\begin{lemma}\label{2}
 Given $r,s\in\mathbb N$, there exists $l=l(r,s)\in\mathbb N$ such that whenever $A$ is an $IP^*_r$ set in $\mathcal P_f(\mathbb N)$ and $B$ is an $IP^*_s$ set in $\mathcal P_f(\mathbb N)$, one has that $A\cap B$ is an $IP^*_l$ set in $\mathcal P_f(\mathbb N)$. 

\end{lemma}
\begin{proof}
\cite[Propostion 2.5]{b}.
\end{proof}

We also need the following lemma.
\begin{lemma}\label{1}  Let $(S,\cdot )$ be a semigroup, let $A$ be a $CR$ set, and let $F\in\mathcal P_f(\mathbb N_S)$. Let $\Theta=\Theta (A,F)=\{L\in\mathcal P_f(\mathbb N):\text { if } m=|L|\text{ and } L=\{t(1),t(2),\ldots,t(m)\}_<,$ then $(\exists\, a\in S^{m+1})(\forall  f\in F)$ $(a(1)\cdot f(t(1))\cdot a(2)\cdots a(m)\cdot f(t(m))\cdot a(m+1)\in A)\}$. If $k\in\mathbb N,|F|\leq k$, and $r=r(A,K)$, then $\Theta$ is an $IP^*_r$ set in $\mathcal P_f(\mathbb N)$.  
\end{lemma}

\begin{proof} Let $\langle H_n\rangle^r_{n=1}$
be a sequence in $\mathcal P_f(\mathbb N)$ such that for each $n\in \{1,2,\ldots,r-1\}$, $\max H_n<\min H_{n+1}$. For $n\in\{1,2,\ldots,r\}$, let $\alpha_n=|H_n|$ and write $H_n=\{b(n,1),b(n,2),\ldots,b(n,\alpha_n)\}_<$. Pick $d\in S$ and for $f\in F$, define $g_f\in\mathbb N_S$ by 
\[
g_f(n)=f(b(n,1))\cdot d\cdot f(b(n,2))\cdot d\cdots d\cdot f(b(n,\alpha_n)) \text { if } n\in \{1,2,\cdots,r\},
\]
and $g_f(n)=d$ otherwise. Now $\{g_f:f\in F\}\in \mathcal P_f(\mathbb N_S)$ with $|\{g_f:f\in F\}|\leq k$. So pick $m\in\mathbb N$, $a\in S^{m+1}$, and $t(1)<t(2)<\cdots<t(m)\leq r$ in $\mathbb N$, such that for each $f\in F,a(1)\cdot g_f(t(1))\cdot a(2)\cdots a(m)\cdot g_f (t(m))\cdot a(m+1)\in A$.

 We claim that $\bigcup^m_{j=1}H_{t(j)}\in\Theta$. Let $p=\sum^m_{j=1}\alpha_{t(j)}$. Then $p=|\bigcup^m_{j=1}H_{t(j)}|$ and the elements of $\bigcup^m_{j=1}H_{t(j)}$ in order are 
 \begin{align*}
 s(1)<s(2)<&\dots<s(p)=\\
b(t(1),1)<b(t(1),2)<&\dots<b(t(1),\alpha_{t(1)})<\\ b(t(2),1)<b(t(2),2)<&\dots<b(t(2),\alpha_{t(2)})<  \\
&\vdots\\
b(t(m),1)<b(t(m),2)<&\dots<b(t(m),\alpha_{t(m)}).  
 \end{align*}
 
Define $c\in S^{p+1}$ by $c(i)=a(j)$ if $j\in\{1,2,\ldots,m\}$ and $s(i)=b(t(j),1)$, $c(p+1)=a(m+1)$, and $c(i)=d$ if $i\in\{1,2,\ldots,p\}$ and $s(i)\neq b(t(j),1)$ for any $j\in\{1,2,\ldots,m\}$. Then for each $f\in F$,
\begin{align*}
 &c(1)\cdot f(s(1))\cdot c(2)\cdots c(p)\cdot f(s(p))\cdot c(p+1)=\\
 &a(1)\cdot g_f(t(1))\cdot a(2)\cdots
a(m)\cdot g_f(t(m))\cdot a(m+1)\in A.
\end{align*}

Thus $\bigcup^m_{j=1}H_{t(j)}\in\Theta$ as required. 
\end{proof}

Now we are in the position to prove our main theorem.
\begin{proof}[\textbf{Proof of Theorem \ref{m}:}]

To see that $A\times B$ is a CR set, let $k\in \mathbb N$. Let $u=r(A,K)$, let $v=r(B,k)$, and let $l=l(u,v)$ be as guaranteed by Lemma \ref{1}, Let $F\in\mathcal P_f(\mathbb N_{(S\times T)})$ with $|F|\leq K$. We will show that there exist $m\in\mathbb N,c\in(S\times T)^{m+1}$, and $t(1)<t(2)<\cdots<t(m)\leq l$ such that for each $f\in F$,
 \[
c(1)\cdot f(t(1))\cdot c(2)\cdots c(m)\cdot f(t(m))\cdot c(m+1)\in A\times B.
 \]

Let $G=\{\pi_1\circ f:f\in F\}$ and let $H=\{\pi_2\circ f:f\in F\}$. Let $\Theta_1=\Theta(A,G)$ and let $\Theta_2=\Theta(B,H)$. By Lemma \ref{1}, $\Theta_1$ is an $IP^*_u$ set in $\mathcal P_f(\mathbb N)$ and $\Theta_2$ is an $IP^*_v$ set in $\mathcal P_f(\mathbb N)$ so by Lemma \ref{2}, $\Theta_1\cap\Theta_2$ is an $IP^*_l$ set in $\mathcal P_f(\mathbb N)$. Consequently $\Theta_1\cap\Theta_2\cap FU(\langle\{i\}\rangle^l_{i=1})\neq\emptyset$. Pick $L\in\Theta_1\cap\Theta_2\cap FU(\langle\{i\}\rangle^l_{i=1})$, let $m=|L|$, and write $L=\{t(1),t(2),\ldots t(m)\}_<$. Since $L\in FU(\langle\{i\}\rangle^l_{i=1})=\{1,2,\cdots,l\}$, we have that $t(m)\leq l$.

 Pick $a\in S^{m+1}$ such that for each $f\in F$,
 \[
 a(1)\cdot \pi_1\circ f(t(1))\cdot a(2)\cdots a(m)\cdot\pi_1\circ f(t(m))\cdot a(m+1)\in A
 \]
 and pick $b\in T^{m+1}$ such that for each $f\in F$,
 \[
 b(1)\cdot\pi_2\circ f(t(1))\cdot b(2)\cdots b(m)\cdot \pi_2\circ f(t(m))\cdot b(m+1)\in B.
 \]
 For $j\in\{1,2,\ldots,m+1\}$ let $c(j)=(a(j),b(j))$. Then for $f\in F$,
 \[
 c(1)\cdot f(t(1))\cdot c(2)\cdots c(m)\cdot f(t(m))\cdot c(m+1)\in A\times B
 \]
 as required.

\end{proof}

\section*{Acknowledgement} The Author is thankful to the referee for his/her comments on the previous draft of the paper which helped us to improve this draft. The author is also supported by NBHM postdoctoral fellowship with reference no: 0204/27/(27)/2023/R \& D-II/11927. 

\section*{Declaration of Interest Statement} The authors declare that they have no known competing financial interests or personal relationships that could have appeared to influence the work reported in this article.

\end{document}